\documentclass[10pt, a4paper]{amsart}

\usepackage{amsfonts,amsmath,amssymb, amscd}
\usepackage{pinlabel}

\newtheorem{theorem}{Theorem}[section]

\newtheorem{proposition}[theorem]{Proposition}
\newtheorem{corollary}[theorem]{Corollary}

\theoremstyle{definition}
\newtheorem{remark}[theorem]{Remark}

\newcommand\ZZ{\mathbb{Z}}
\newcommand\RR{\mathbb{R}}
\newcommand\Aut{\operatorname{Aut}}

\newcommand\GL{\operatorname{GL}}
\newcommand\SL{\operatorname{SL}}
\newcommand\Sp{\operatorname{Sp}}
\newcommand\St{\operatorname{St}}

\DeclareMathOperator{\id}{id}

\numberwithin{equation}{section}

\title[On an action of the braid group $B_{2g+2}$ on the free group~$F_{2g}$]
{On an action of the braid group $B_{2g+2}$\\ 
on the free group~$F_{2g}$}

\author{Christian Kassel}
\address{Christian Kassel: 
Institut de Recherche Math\'e\-ma\-tique Avanc\'ee,
CNRS \& Universit\'e de Strasbourg,
7 rue Ren\'{e} Descartes, 67084 Strasbourg, France}
\email{kassel@math.unistra.fr}
\urladdr{www-irma.u-strasbg.fr/\raise-2pt\hbox{\~{}}kassel/}

\dedicatory{En mi'm alte Kumpel Christophe Reutenauer}

\keywords{Braid group, free group, symplectic group, ramified covering}

\subjclass[2010]{Primary 20F36, 20E05, 20H05; Secondary 11F46, 57M07, 57M12}

\begin{document}

\begin{abstract}
We construct an action of the braid group~$B_{2g+2}$ on the free group~$F_{2g}$
extending an action of~$B_4$ on~$F_2$ introduced earlier by Reutenauer and the author.
Our action induces a homomorphism from~$B_{2g+2}$ into the symplectic modular group~$\Sp_{2g}(\ZZ)$.
In the special case $g=2$ we show that the latter homomorphism is surjective 
and determine its kernel, thus obtaining a braid-type presentation of~$\Sp_4(\ZZ)$.
\end{abstract}

\maketitle

\section{Introduction}

In~\cite{KR} Christophe Reutenauer and the present author considered 
the automorphisms $G$, $D$, $\widetilde{G}$, $\widetilde{D}$
of the free group~$F_2$ on two generators $a$ and~$b$ defined by
\begin{equation}\label{def-R}
\begin{matrix}
G: (a,b) \mapsto (a,ab) \, , &\qquad& D: (a,b) \mapsto (ba,b) \, , \\
\widetilde{G}: (a,b) \mapsto (a,ba) \, , &\qquad& \widetilde{D}: (a,b) \mapsto (ab,b) \, ,
\end{matrix}
\end{equation}
(see also~\cite[Sect.\ 2.2.2]{Lo}).
Their images under the natural surjection (the abelianization map) 
$\pi: \Aut(F_2) \to \Aut(\ZZ^2) = \GL_2(\ZZ)$ 
are the matrices
\begin{equation}\label{AB}
\pi(G) = \pi(\widetilde{G}) = A = 
\begin{pmatrix}
1 &  1 \\
0 & 1
\end{pmatrix}
\quad\text{and}\quad
\pi(D) = \pi(\widetilde{D}) = B = 
\begin{pmatrix}
1 & 0 \\
1 & 1
\end{pmatrix} .
\end{equation}
The matrices $A$ and $B$ generate the 
subgroup~$\SL_2(\ZZ)$
and satisfy the braid relation
\begin{equation*}
A B^{-1} A =  B^{-1} A B^{-1} \, .
\end{equation*}

In \cite[Lemma\,2.1]{KR} we observed that 
$G$, $D$, $\widetilde{G}$, $\widetilde{D}$ 
satisfy similar braid relations in the automorphism group~$\Aut(F_2)$, namely
\begin{equation*}
G D^{-1} G  = D^{-1} G D^{-1}\, , \quad
\widetilde{G} D^{-1} \widetilde{G}  = D^{-1} \widetilde{G} D^{-1}
\end{equation*}
\begin{equation*}
\widetilde{G} \widetilde{D}^{-1} \widetilde{G}  = \widetilde{D}^{-1} \widetilde{G} \widetilde{D}^{-1}\, , \quad
G \widetilde{D}^{-1} G = \widetilde{D}^{-1} G \widetilde{D}^{-1} \, ,
\end{equation*}
together with the commutation relations
\begin{equation*}
G \widetilde{G} = \widetilde{G} G 
\quad\text{and}\quad 
D \widetilde{D} = \widetilde{D} D \, .
\end{equation*}
These relations allowed us to define a group homomorphism $f$ from the braid group~$B_4$ on four strands
to~$\Aut(F_2)$ by
\begin{equation}\label{def-f}
f(\sigma_1) = G \, , \quad f(\sigma_2) = D^{-1} \, , \quad f(\sigma_3) = \widetilde{G} \, ,
\end{equation}
where $\sigma_1, \sigma_2, \sigma_3$ are the standard generators of~$B_4$.
%(defined e.g.\ in~\cite{Bi2} or in~\cite[Sect.\,1.1]{KT}).
In~\cite{KR} we proved that the image $f(B_4)$ of~$f$ is the 
index-two subgroup  
$\pi^{-1}(\SL_2(\ZZ))$ of~$\Aut(F_2)$
and that the kernel of~$f$ is the center of~$B_4$.

%Besides, $G$, $D$, $\widetilde{G}$, $\widetilde{D}$ 
%generate a submonoid~$\St_0$ of~$\Aut(F_2)$;
%this monoid is called the \emph{special Sturmian monoid}.
%In~\cite[Sect.\,3]{KR} we gave a monoid presentation of~$\St_0$
%and proved that $\St_0$ is isomorphic to the submonoid of~$B_4$ generated by the four braids
%$\sigma_1, \sigma_2^{-1}, \sigma_3$,
%and $(\sigma_1 \sigma_3^{-1}) \, \sigma_2^{-1} \, (\sigma_1 \sigma_3^{-1})^{-1}$.

After the article~\cite{KR} was published, 
Etienne Ghys suggested that the action  of~$B_4$ on~$F_2$ given by\,\eqref{def-f} might be derived 
from the fact that a punctured torus is a double covering of a disk branched over four points.
We checked that this fact indeed led to\,\eqref{def-f}. 

The present article is a continuation of~\cite{KR};
it is based on the fact that a punctured surface of arbitrary genus~$g$ can be realized
as a ramified double covering of a disk with $2g+2$ ramification points. 
Our main result yields
an action of the braid group~$B_{2g+2}$ on $2g+2$~strands on the free group~$F_{2g}$ on $2g$~generators:
this action is given by an explicit group homomorphism
$f: B_{2g+2} \to \Aut(F_{2g})$, extending\,\eqref{def-f}.

The formulas for this homomorphism are given in Section~\ref{main};
we show in Section~\ref{geom} how to derive them geometrically.
In~Section~\ref{monoid}, in an attempt to define higher analogues of the \emph{special Sturmian monoid}, 
introduced in~\cite{KR},
we search for automorphisms in $f(B_{2g+2}) \subset \Aut(F_{2g})$ 
that preserve the free monoid on~${2g}$ generators. 
It turns out that the situation for $g\geq 2$ is less satisfactory than for $g=1$.

Finally, in Section~\ref{image} we first observe that the image $f(B_{2g+2})$ maps
under the abelianization map $\pi: \Aut(F_{2g}) \to \GL_{2g}(\ZZ)$ 
into the symplectic modular group~$\Sp_{2g}(\ZZ)$.
Concentrating on the case $g=2$, we show that the map
$\pi \circ f : B_6 \rightarrow \Sp_{4}(\ZZ) $
is surjective and we determine its kernel;
we thus obtain a braid-type presentation of~$\Sp_4(\ZZ)$ 
with generators the standard generators of~$B_6$ 
and with relations the usual braid relations 
together with four additional relations.

\section{The main result}\label{main}

Let $g$ be an integer $\geq 1$ and $F_{2g}$ be the \emph{free group} on $2g$~generators
$a_1, \ldots, a_g$, $b_1, \ldots, b_g$. 

\subsection{A family of automorphisms of~$F_{2g}$}

We consider the $2g+1$ automorphisms $u_1, \ldots, u_{2g+1}$ of~$F_{2g}$ defined as follows.

\begin{itemize}

\item 
The automorphism~$u_1$ fixes all generators, except~$b_1$ for which we have
\begin{equation}\label{u1}
u_1(b_1) = a_1 b_1 \, .
\end{equation}

\item
The automorphism~$u_{2g+1}$ fixes all generators, except~$b_g$ for which
\begin{equation}\label{ulast}
u_{2g+1}(b_g) = b_g a_g \, .
\end{equation}

\item
For $i = 1, \ldots, g$, the automorphism~$u_{2i}$ fixes all generators, except~$a_i$ for which
\begin{equation}\label{ueven}
u_{2i}(a_i) = b_i^{-1} a_i \, .
\end{equation}

\item
For $i = 1, \ldots, g-1$, the automorphism~$u_{2i+1}$ fixes all generators, except~$b_i$ 
and~$b_{i+1}$ for which we have
\begin{equation}\label{uodd}
u_{2i+1}(b_i) = b_i a_i a_{i+1}^{-1}
\quad\text{and}\quad
u_{2i+1}(b_{i+1}) = a_{i+1} a_i^{-1} b_{i+1}\, .
\end{equation}

\end{itemize}
%In Section~\ref{geom} we shall derive Formulas\,\eqref{u1}--\eqref{uodd} from the consideration of Dehn twists 
%on a surface of genus~$g$.

When $g= 1$, the above $2g+1$ automorphisms reduce to three, namely $u_1, u_2, u_3$; these coincide 
respectively with the automorphisms $G, D^{-1}, \widetilde{G}$ of~$F_2$ defined by\,\eqref{def-R}.

\subsection{An action of the braid group~$B_{2g+2}$}

Let $B_{2g+2}$ be the {braid group} on $2g+2$~strands
with its standard presentation by generators $\sigma_1, \ldots, \sigma_{2g+1}$
satisfying the relations (where $1 \leq i,j \leq 2g+1$)
\begin{equation}\label{braid1}
\sigma_i \sigma_j  = \sigma_j  \sigma_i  \qquad \text{if} \; |i-j| > 1
\end{equation}
and 
\begin{equation}\label{braid2}
\sigma_i \sigma_j \sigma_i  = \sigma_j \sigma_i \sigma_j \quad \text{if} \; |i-j| = 1 \, .
\end{equation}

We now state our main result:
it yields an action of~$B_{2g+2}$ on the free group~$F_{2g}$ by group automorphisms.

\begin{theorem}\label{thm-main}
There is a group
homomorphism $f : B_{2g+2} \to \Aut(F_{2g})$ such that
$f(\sigma_i) = u_i$ for all $i = 1, \ldots, 2g+1$.
\end{theorem}

The proof is straightforward: it suffices to check that the automorphisms
$u_1, \ldots,$ $u_{2g+1}$ satisfy Relations~\eqref{braid1} and~\eqref{braid2}.
(Formulas~\eqref{u1}--\eqref{uodd} and Theorem\,\ref{thm-main} 
were first publicized in~\cite[Exercise\,1.5.2]{KT}.)

For $g=1$, the homomorphism of Theorem~\ref{thm-main} 
coincides with the homomorphism $f: B_4 \to \Aut(F_2)$
of~\cite[Lemma\,2.5]{KR} defined in the introduction. 
In \emph{loc.\ cit}.\ we showed that its kernel is exactly the center of~$B_4$.
For $g>1$ we have the following weaker result.

\begin{proposition}\label{vanish}
The kernel of $f : B_{2g+2} \to \Aut(F_{2g})$ contains the center of~$B_{2g+2}$. 
\end{proposition}

\begin{proof}
Let $\delta = \sigma_1\cdots \sigma_{2g+1}$. It is well known (see \cite{Bi2, KT})
that the center of~$B_{2g+2}$ is generated by~$\delta^{2g+2}$. 
Let $\widetilde{\delta} = f(\delta) = u_1 \cdots u_{2g+1}$ be the corresponding automorphism
of~$F_{2g}$. Using~\eqref{u1}--\eqref{uodd}, it is easy to check that
for all $i= 1, \ldots, g$ we have
\begin{equation*}
\widetilde{\delta}(a_i) = (b_1 \ldots b_i)^{-1}
\quad\text{and}\quad
\widetilde{\delta}(b_i) = 
\begin{cases}
a_i a_{i+1}^{-1} \;\;\; \text{if}\; i \neq g \, ,\\
a_g \qquad \;\; \text{if}\; i = g \, .
\end{cases}
\end{equation*} 
We deduce that
\begin{equation*}
\widetilde{\delta}^2(a_i) = 
\begin{cases}
a_{i+1}a_1^{-1} \;\;\; \text{if}\; i \neq g \, ,\\
a_1^{-1} \qquad \;\;\, \text{if}\; i = g \, ,
\end{cases}
\;\text{and}\quad
\widetilde{\delta}^2(b_i) = 
\begin{cases}
b_{i+1} \qquad\qquad \text{if}\; i \neq g \, ,\\
(b_1 \ldots b_g)^{-1} \;\;\,  \text{if}\; i = g \, .
\end{cases}
\end{equation*}
From these formulas it is easy to conclude that $\widetilde{\delta}^{2g+2}$ is the identity.
\end{proof}

\subsection{Preserving the free submonoid~$M_{2g}$ of~$F_{2g}$}\label{monoid}

Let $M_{2g}$ be the submonoid of~$F_{2g}$ generated by $a_1, \ldots, a_g$, $b_1, \ldots, b_g$:
it is a free monoid.
By construction the automorphisms $u_1$ and $u_{2g+1}$ preserve~$M_{2g}$.
So do the inverses~$u_{2i}^{-1}$ of the automorphisms~$u_{2i}$ since
\begin{equation*}
u_{2i}^{-1}(a_i) = b_i a_i \, ,
\end{equation*}
the other generators being fixed.

Unfortunately, as we can see from~\eqref{uodd},
the automorphisms $u_{2i+1}$ ($1 \leq i \leq g-1$), which exist only when $g\geq 2$, 
do not preserve the monoid~$M_{2g}$. Nor do their inverses
since
\begin{equation*}
u_{2i+1}^{-1}(b_i) = b_i a_{i+1} a_i^{-1} 
\quad\text{and}\quad
u_{2i+1}^{-1}(b_{i+1}) = a_i a_{i+1}^{-1} b_{i+1}\, .
\end{equation*}
Thus, the action of the higher braid groups~$B_{2g+2}$ on~$F_{2g}$ with $g\geq 2$ 
is quite different from the action of~$B_4$ on~$F_2$ 
when it comes to preserving the monoid~$M_{2g}$.

Let us consider the case $g=1$.
In~\cite{KR} we observed that the $M_2$-preserving automorphisms $u_1 = G$, 
$u_2^{-1} = D$, $u_3 = \widetilde{G}$ of~$F_2$
are so-called {Sturmian morphisms}.
Together with the Sturmian morphism~$\widetilde{D}$,
the morphisms $G, D, \widetilde{G}$ generate the {special Sturmian monoid}~$\St_0$, 
for which we gave a presentation by generators and relations, 
and which we proved to be isomorphic to the submonoid of~$B_4$ 
generated by 
$\sigma_1$, $\sigma_2^{-1}$, $\sigma_3$, and 
$(\sigma_1 \sigma_3^{-1}) \, \sigma_2^{-1} \, (\sigma_1 \sigma_3^{-1})^{-1}$.
See~\cite[Sect.~3]{KR} for details.

In the case $g \geq 2$, consider the submonoid~$\Omega_{2g}$ of~$\Aut(F_{2g})$ generated by
the $g+2$ elements
$u_1$, $u_{2g+1}$, and $u_{2i}^{-1}$ ($1 \leq i \leq g$). It follows from the observation above that
all elements of~$\Omega_{2g}$ preserve the monoid~$M_{2g}$.

We now express~$\Omega_{2g}$ in terms of the free monoid~$M_2$ on two generators
and the free monoid~$M_1$ on one generator, which we identify with the monoid of non-negative integers.

\begin{proposition}
Let $g\geq 2$.
We have an isomorphism of monoids
\begin{equation*}
\Omega_{2g} \cong M_2 \times (M_1)^{g-2} \times M_2 \, .
\end{equation*}
Moreover, $\Omega_{2g}$ is isomorphic to the submonoid of~$B_{2g+2}$ generated by  
$\sigma_1$, $\sigma_{2g+1}$, and $\sigma_{2i}^{-1}$ ($1 \leq i \leq g$).
\end{proposition}

\begin{proof}
In view of the commutation relations~\eqref{braid1} for the automorphisms~$u_i$,
any product of $u_1$, $u_2^{-1}$, $u_4^{-1}$, \ldots, $u_{2g-2}^{-1}$, 
$u_{2g}^{-1}$, $u_{2g+1}$ can be uniquely written as
\begin{equation*}
w(u_1, u_2^{-1}) \, u_4^{-n_2} \cdots u_{2g-2}^{-n_{g-1}}  \, w'(u_{2g}^{-1},u_{2g+1}) \, ,
\end{equation*}
where $w(u_1, u_2^{-1})$ belongs to the submonoid of~$\Aut(F_{2g})$ generated by  
$u_1$ and $u_2^{-1}$, the exponents $n_2, \ldots, n_{g-1}$ are non-negative integers, and
$w'(u_{2g}^{-1}, u_{2g+1})$ belongs to the submonoid of~$B_{2g+2}$ generated by $u_{2g}^{-1}$
and $u_{2g+1}$.
It remains to show that the submonoid generated by $u_1$, $u_2^{-1}$ 
and the submonoid generated by $u_{2g}^{-1}$, $u_{2g+1}$  are both isomorphic to the free monoid~$M_2$.
We give a proof of this claim for the first submonoid (there is a similar proof for the second one). 
Since $u_1$ and $u_2^{-1}$ move only the generators~$a_1$ and $b_1$, we may consider them in~$\Aut(F_2)$.
Now let $w(u_1, u_2^{-1})$ be a non-trivial word in $u_1$, $u_2^{-1}$ and consider its image in~$\GL_2(\ZZ)$;
we have
\begin{equation*}
\pi \left(w(u_1, u_2^{-1}) \right) = w(A, B) \, 
\end{equation*}
where $A$ and $B$ are the matrices defined by~\eqref{AB}. It is well known (and easy to check) that
any non-trivial word in~$A$, $B$ cannot be the identity matrix.
Therefore, $w(u_1, u_2^{-1}) \neq 1$ in~$\Aut(F_2)$.
This proves our claim.

Set $\iota(u_1) = \sigma_1$, $\iota(u_{2g+1}) = \sigma_{2g+1}$, and $\iota(u_{2i}^{-1}) = \sigma_{2i}^{-1}$
for $i = 1, \ldots, g$.
In view of the first assertion and the braid relations~\eqref{braid1}--\eqref{braid2}, 
these formulas define a monoid homomorphism $\iota : \Omega_{2g} \to B_{2g+2}$.
Since $f \circ \iota = \id$ on~$\Omega_{2g}$, the homomorphism $\iota$ is injective, 
which proves the second assertion.
\end{proof}

\section{Ramified double coverings of the disk}\label{geom}

We now explain how we found Formulas\,\eqref{u1}--\eqref{uodd} which define the automorphisms
$u_1, \ldots, u_{2g+1}$ of Section\,\ref{main}.
The material in this section is standard; we nevertheless give details for the sake of non-topologists.

It is well known that any closed surface~$\Sigma_g$ of genus~$g>0$ can be realized 
as a ramified double covering of the sphere~$S^2$ with $2g+2$ ramification points.
It suffices to embed~$\Sigma_g$ into~$\RR^3$ in such a way that it is invariant under
the \emph{hyperelliptic involution}, which is the reflexion in a line~$L$
intersecting~$\Sigma_g$ in $2g+2$ points as in Figure\,\ref{fig1}.
The quotient of~$\Sigma_g$ by this involution is a sphere equipped with $2g+2$ distinguished points,
namely the projections of the points of~$\Sigma_g \cap L$; these are
the ramification points of the double covering.

\begin{figure}[ht!]
\labellist
\small\hair 2pt
\pinlabel $D$ [r] at 273 153
\pinlabel $\ast$ [r] at 295 142
\pinlabel $D'$ [r] at 273 40
\pinlabel $a_i$ [r] at 338 123
\pinlabel $b_i$ [r] at 483 155
\pinlabel {\scriptsize{$i$-th hole}} [r] at 423 95
\pinlabel {\scriptsize{first hole}} [r] at 196 95
\pinlabel $L$ [r] at 0 95
\endlabellist
\centering
\includegraphics[scale=0.45]{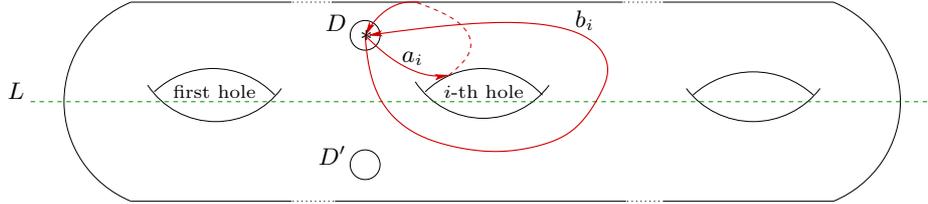}
\caption{A surface~$\Sigma_g$ invariant under the hyperelliptic involution.}
\label{fig1}
\end{figure}

From the interior of~$\Sigma_g - \Sigma_g \cap L$ remove
a small open disk~$D$ with center~$P$ (represented by~$\ast$ in the figures), 
as well as the disk~$D'$
obtained from~$D$ under the reflection in~$L$.
Then $\Sigma_g^{\circ} = \Sigma_g - (D \cup D')$ is a ramified double covering of the sphere deprived of a disk, 
in other words, a double covering of a disk~$D_0$ with $2g+2$ ramification points. 

As is well known (see \cite{Bi2} or~\cite[Sect.~1.6]{KT}),
the braid group~$B_{2g+2}$ is isomorphic to the mapping class group of~$D_0$
consisting of the isotopy classes of orientation-preserving homeomorphisms
that fix each point of the boundary of~$D_0$ and permute the $2g+2$ distinguished points.
If $\varphi$ is such an homeomorphism, 
we pick the lift~$\widetilde{\varphi}$ of~$\varphi$ to~$\Sigma_g^{\circ}$ that fixes $D$ and~$D'$ pointwise.
The correspondence $\varphi \mapsto \widetilde{\varphi}$ induces a homomorphism
from~$B_{2g+2}$ to the mapping class group of~$\Sigma_g^{\circ}$, hence
a homomorphism $B_{2g+2} \to \Aut(\Pi)$, 
where $\Pi$ is the fundamental group of~$\Sigma_g^{\circ}$, which is a free group.
We wish to determine this homomorphism.

It is enough to determine the action of the generators of~$B_{2g+2}$ on a smaller
free group, namely the fundamental group of $\Sigma_g - D'$, whose elements we represent
as loops in $\Sigma_g - D'$ based at the point~$P$.
This fundamental group is the free group~$F_{2g}$ 
generated by the loops~$a_1, \ldots, a_g$, $b_1, \ldots, b_g$ as depicted in Figure~\ref{fig1}.

\begin{figure}[ht!]
\labellist
\small\hair 2pt
\pinlabel $C_0$ [r] at 63 53
\pinlabel $C_1$ [r] at 253 53
\pinlabel $C_2$ [r] at 413 53
\pinlabel $C_g$ [r] at 683 53
\pinlabel $C'_1$ [r] at 163 153
\pinlabel $C'_2$ [r] at 333 153
\pinlabel $C'_g$ [r] at 583 153
\endlabellist
\centering
\includegraphics[scale=0.45]{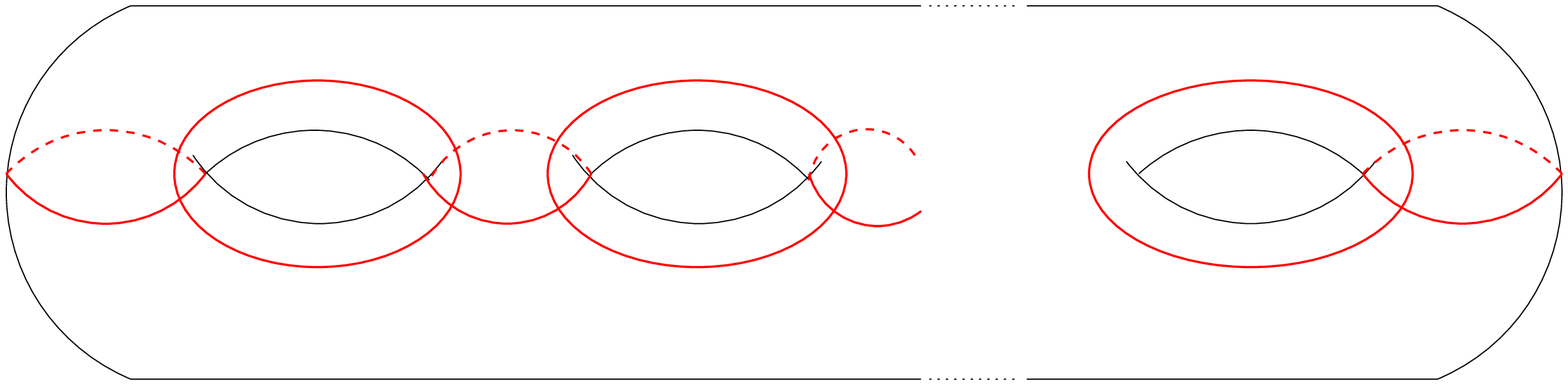}
\caption{The curves $C_i$ and $C'_i$.}
\label{fig2}
\end{figure}

Consider the curves $C_0, C_1, \ldots, C_g$, $C'_1, \ldots, C'_g$ of Figure~\ref{fig2}.
It is easy to check that a lift~$\widetilde{\sigma_1}$ of the homeomorphism of~$D_0$
representing the generator~$\sigma_1$ of~$B_{2g+2}$
is the Dehn twist~$T_0$ along the curve~$C_0$
(for a definition of Dehn twists, see~\cite[Sect.\,3.2.4]{KT}).
The action of~$T_0$ on the generators of the fundamental group of~$\Sigma_g - D'$
is easy to compute; clearly it leaves the generators $a_1, \ldots, a_g$
as well as $b_2, \ldots, b_g$ unchanged. 
On~$b_1$ it acts as in Figure~\ref{fig3}, which leads to Formula~\eqref{u1}.

\begin{figure}[ht!]
\labellist
\small\hair 2pt
\pinlabel $P$ [r] at 389 206
\pinlabel $\ast$ [r] at 366 190
\pinlabel {\scriptsize{first hole}} [r] at 267 122
\endlabellist
\centering
\includegraphics[scale=0.45]{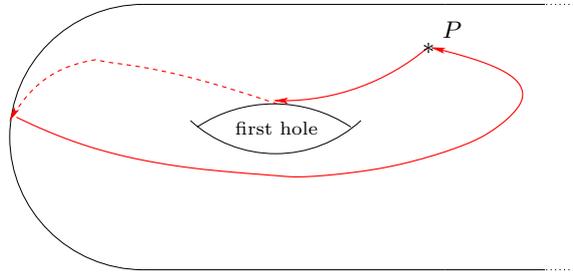}
\caption{Action of the Dehn twist $T_0$ on~$b_1$.}
\label{fig3}
\end{figure}

Similarly, a lift~$\widetilde{\sigma_{2g+1}}$ of the homeomorphism representing the last generator~$\sigma_{2g+1}$ 
of~$B_{2g+2}$ is the Dehn twist~$T_g$ along the curve~$C_g$. 
This twist fixes all generators of the fundamental group except~$b_g$.
The reader is encouraged to draw the corresponding figure and derive~\eqref{ulast} from it.

A lift~$\widetilde{\sigma_{2i}}$ 
of the homeo\-morphism representing the generator~$\sigma_{2i}$ ($1\leq i \leq  g$)
is the Dehn twist~$T'_i$ along the curve~$C'_i$. 
This twist affects only the generator~$a_i$, on which it acts as in Figure~\ref{fig4};
we thus obtain~\eqref{ueven}.

\begin{figure}[ht!]
\labellist
\small\hair 2pt
\pinlabel $P$ [r] at 298 214
\pinlabel $\ast$ [r] at 277 198
\pinlabel {\scriptsize{$i$-th hole}} [r] at 210 123
\endlabellist
\centering
\includegraphics[scale=0.45]{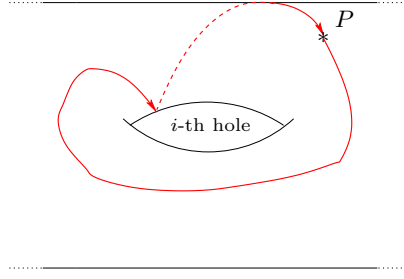}
\caption{Action of the Dehn twist $T'_i$ on~$a_i$ ($1\leq i \leq  g$).}
\label{fig4}
\end{figure}

Finally, when $1\leq i \leq  g-1$, a lift~$\widetilde{\sigma_{2i+1}}$ of the homeomorphism 
representing~$\sigma_{2i+1}$ is the Dehn twist~$T_i$ along the curve~$C_i$. 
This twist fixes all generators of the fundamental group except~$b_i$ and~$b_{i+1}$;
it acts on~$b_i$ as in Figure~\ref{fig5} and on~$b_{i+1}$ as in Figure~\ref{fig6}.
This yields Formula~\eqref{uodd}.

\begin{figure}[ht!]
\labellist
\small\hair 2pt
\pinlabel $P$ [r] at 345 203
\pinlabel $\ast$ [r] at 324 185
\pinlabel {\tiny{$i$-th hole}} [r] at 193 123
\pinlabel {\tiny{$i+1$-st hole}} [r] at 502 122
\endlabellist
\centering
\includegraphics[scale=0.45]{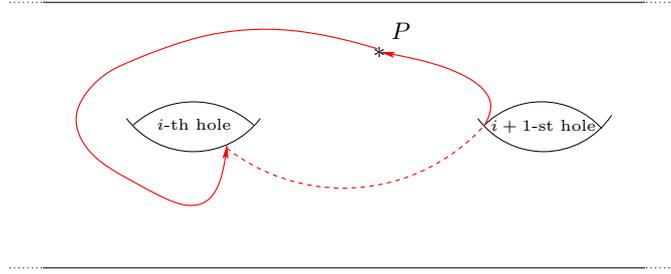}
\caption{Action of the Dehn twist $T_i$ on~$b_i$.}
\label{fig5}
\end{figure}

\begin{figure}[ht!]
\labellist
\small\hair 2pt
\pinlabel $P$ [r] at 308 208
\pinlabel $\ast$ [r] at 298 185
\pinlabel {\tiny{$i$-th hole}} [r] at 163 121
\pinlabel {\tiny{$i+1$-st hole}} [r] at 473 121
\endlabellist
\centering
\includegraphics[scale=0.45]{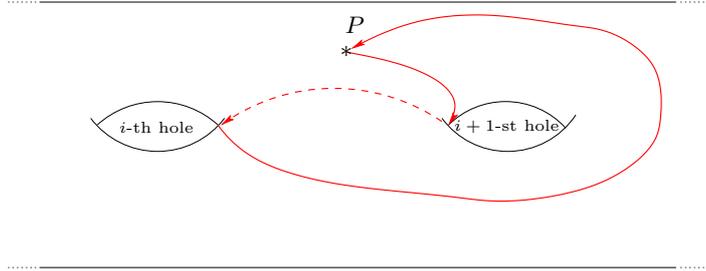}
\caption{Action of the Dehn twist $T_i$ on~$b_{i+1}$.}
\label{fig6}
\end{figure}

The above computation appeared in~\cite{Bi1} with different notation; the correspondence
between that paper and ours is given by $s_i = b_i^{-1}$, $t_i = a_i^{-1}$, 
$h_{C_1} = T_0^{-1}$, $h_{U_i} = T'_i{}^{-1}$, and $h_{Z_i} = T_i^{-1}$
(see also~\cite[Sect.~4]{Go} and~\cite[Sect.~2]{BM}).
A~proof of Proposition~\ref{vanish} can also be derived from Relation~(10) in~\cite{Bi1}.

\section{The symplectic modular group}\label{image}

In this section we first observe that, after abelianizing our action, we obtain
a symplectic action of the braid group~$B_{2g+2}$ on the free abelian group~$\ZZ^{2g}$.
In the second part of the section we elaborate on the case $g=2$
and obtain a braid-type presentation of~$\Sp_4(\ZZ)$.

\subsection{Symplectic automorphisms}

Let $g \geq 1$ be a positive integer. 
Pick a basis $(\bar{a}_1, \ldots, \bar{a}_g, \bar{b}_1, \ldots, \bar{b}_g)$ of~$\ZZ^{2g}$.
Using this basis, we identify the group~$\Aut(\ZZ^{2g})$ of automorphisms of~$\ZZ^{2g}$
with the general linear group~$\GL_{2g}(\ZZ)$.

We equip $\ZZ^{2g}$ with the standard alternating bilinear form $\langle \; ,\, \rangle$
which vanishes on all pairs of basis elements, except the following ones:
\begin{equation*}\label{alternating}
\langle \bar{a}_i, \bar{b}_i \rangle = - \langle \bar{b}_i, \bar{a}_i \rangle = 1
\qquad (i= 1, \ldots, g) \, .
\end{equation*}
The \emph{symplectic modular group}~$\Sp_{2g}(\ZZ)$
(formerly called Siegel's modular group) is the
group of automorphisms of~$\ZZ^{2g}$ preserving this alternating form;
it can be described as the group of matrices $M \in \GL_{2g}(\ZZ)$ such that
\begin{equation}\label{def-sympl}
M^T 
\begin{pmatrix}
0 & I_g \\
- I_g & 0
\end{pmatrix}
M = 
\begin{pmatrix}
0 & I_g \\
- I_g & 0
\end{pmatrix} ,
\end{equation}
where $M^T$ is the transpose of~$M$ and $I_g$ is the identity matrix of size~$g$.

Consider the composition
\begin{equation*}\label{fbar}
\overline{f} : B_{2g+2} \to \GL_{2g}(\ZZ)
\end{equation*}
of the homomorphism $f : B_{2g+2} \to \Aut(F_{2g})$ of Section~\ref{main} with the 
natural surjection
$\pi : \Aut(F_{2g}) \to \GL_{2g}(\ZZ)$.

\begin{proposition}
We have $\overline{f}(B_{2g+2}) \subset \Sp_{2g}(\ZZ)$.
\end{proposition}

\begin{proof}
It is enough to check that the image $\pi(u_i)$ in~$\GL_{2g}(\ZZ)$
of each automorphism~$u_i$ belongs to~$\Sp_{2g}(\ZZ)$.
Now the automorphisms~$u_i$ are induced by elements of a mapping class group 
which are well known to induce symplectic linear maps (see~\cite[Sect.~5.8]{MKS}).
Alternatively, one checks that each matrix~$\pi(u_i)$ satisfies
Relation\,\eqref{def-sympl}.
\end{proof}
 
When $g=1$, the symplectic group $\Sp_2(\ZZ)$ identifies naturally with the modular group~$\SL_2(\ZZ)$.
We proved in~\cite{KR} that the image $\overline{f}(B_4)$ is the entire
group~$\SL_2(\ZZ)$. 

\begin{remark}
Taking an adequate specialization of the Burau representation, 
Magnus and Peloso~\cite{MP} also constructed a homomorphism $B_{2g+2} \to \Sp_{2g}(\ZZ)$, 
for which they showed that it is surjective if and only if $g = 1$ or $g = 2$. 
We do not know if their symplectic representation is related to ours.
\end{remark}

\subsection{The case $g=2$}

We now consider the homomorphism $\overline{f} : B_{2g+2} \rightarrow \Sp_{2g}(\ZZ)$
in the special case~$g = 2$.

\begin{theorem}\label{thm-sp4}
The homomorphism $\overline{f} : B_6 \rightarrow \Sp_4(\ZZ)$ is surjective and 
its kernel is the normal subgroup of~$B_6$ generated by~$\Delta^2$, $\alpha^2$, $\alpha\beta$, $(\alpha\gamma)^2$, 
where
\begin{equation*}
\Delta =   \sigma_ 1\sigma_2 \sigma_3 \sigma_4 \sigma_5 \sigma_1 \sigma_2 \sigma_3 \sigma_4
\sigma_1 \sigma_2 \sigma_3 \sigma_1 \sigma_2 \sigma_1 \, ,
\end{equation*}
\begin{equation*}
\alpha = (\sigma_4 \sigma_5)^3 \, ,
\qquad
\beta = \sigma_3^{-1}(\sigma_1 \sigma_2)^3 \sigma_3 \, ,
\qquad
\gamma = \sigma_1 \sigma_3^{-1} \sigma_5 \, .
\end{equation*}
\end{theorem}

As a consequence, we obtain the following braid-type presentation of~$\Sp_4(\ZZ)$.

\begin{corollary}\label{coro-sp4}
The symplectic group $\Sp_4(\ZZ)$ has a presentation with generators $\sigma_1, \sigma_2, \sigma_3, \sigma_4, \sigma_5$
and $14$~relations consisting of Relations~\eqref{braid1}--\eqref{braid2} and of the four relations
\begin{equation*}
(\sigma_ 1\sigma_2 \sigma_3 \sigma_4 \sigma_5 \sigma_1 \sigma_2 \sigma_3 \sigma_4
\sigma_1 \sigma_2 \sigma_3 \sigma_1 \sigma_2 \sigma_1)^2 = 1 \, ,
\end{equation*}
\begin{equation*}
(\sigma_4 \sigma_5)^6 = 1\, ,
\qquad
(\sigma_1 \sigma_2)^3  = \sigma_3(\sigma_4 \sigma_5)^3 \sigma_3^{-1} \, ,
\end{equation*}
\begin{equation*}
(\sigma_1 \sigma_3^{-1} \sigma_5)^{-1} 
= (\sigma_4 \sigma_5)^3 (\sigma_1 \sigma_3^{-1} \sigma_5) (\sigma_4 \sigma_5)^{-3} \, .
\end{equation*}
\end{corollary}

Previously
Behr~\cite{Be} gave a quite different presentation of this group, with six generators and $18$~relations
(see below);
Bender~\cite{Bn} improved Behr's presentation by reducing it to one with two generators and $8$~relations
(see also~\cite{BiS}).

\begin{proof}[Proof of Theorem~\ref{thm-sp4}]
The generators in Behr's presentation~\cite{Be} of~$\Sp_4(\ZZ)$ are the following symplectic matrices:
\begin{equation*}
x_{\beta} =
\begin{pmatrix}
 1 & 0 & 0 & 0\\
 0 & 1 & 0 & 1\\
 0 & 0 & 1 & 0\\
 0 & 0 & 0 & 1
\end{pmatrix},
\quad
x_{\alpha + \beta} =
\begin{pmatrix}
 1 & 0 & 0 & 1\\
 0 & 1 & 1 & 0\\
 0 & 0 & 1 & 0\\
 0 & 0 & 0 & 1
\end{pmatrix},
\quad
x_{2\alpha+\beta} =
\begin{pmatrix}
 1 & 0 & 1 & 0\\
 0 & 1 & 0 & 0\\
 0 & 0 & 1 & 0\\
 0 & 0 & 0 & 1
\end{pmatrix},
\end{equation*}
\begin{equation*}
x_{\alpha} =
\begin{pmatrix}
 1 & 1 & 0 & 0\\
 0 & 1 & 0 & 0\\
 0 & 0 & 1 & 0\\
 0 & 0 & -1 & 1
\end{pmatrix},
\quad
w_{\alpha} =
\begin{pmatrix}
 0 & -1 & 0 & 0\\
 1 & 0 & 0 & 0\\
 0 & 0 & 0 & -1\\
 0 & 0 & 1 & 0
\end{pmatrix},
\quad
w_{\beta} =
\begin{pmatrix}
 1 & 0 & 0 & 0\\
 0 & 0 & 0 & -1\\
 0 & 0 & 1 & 0\\
 0 & 1 & 0 & 0
\end{pmatrix}.
\end{equation*}
The images under~$\overline{f}$ of the elements $\sigma_1, \ldots, \sigma_5$, and~$\Delta$ of~$B_6$
are given by
\begin{equation*}
M_1 = \overline{f}(\sigma_1) = 
\begin{pmatrix}
 1 & 0 & 1 & 0\\
 0 & 1 & 0 & 0\\
 0 & 0 & 1 & 0\\
 0 & 0 & 0 & 1
\end{pmatrix},
\qquad
M_2 = \overline{f}(\sigma_2) = 
\begin{pmatrix}
 1 & 0 & 0 & 0\\
 0 & 1 & 0 & 0\\
 -1 & 0 & 1 & 0\\
 0 & 0 & 0 & 1
\end{pmatrix},
\end{equation*}
\begin{equation*}
M_3 = \overline{f}(\sigma_3) = 
\begin{pmatrix}
 1 & 0 & 1 & -1\\
 0 & 1 & -1 & 1\\
 0 & 0 & 1 & 0\\
 0 & 0 & 0 & 1
\end{pmatrix},
\qquad
M_4 = \overline{f}(\sigma_4) = 
\begin{pmatrix}
 1 & 0 & 0 & 0\\
 0 & 1 & 0 & 0\\
 0 & 0 & 1 & 0\\
 0 & -1 & 0 & 1
\end{pmatrix},
\end{equation*}
\begin{equation*}
M_5 = \overline{f}(\sigma_5) = 
\begin{pmatrix}
 1 & 0 & 0 & 0\\
 0 & 1 & 0 & 1\\
 0 & 0 & 1 & 0\\
 0 & 0 & 0 & 1
\end{pmatrix},
\qquad
M_{\Delta} = \overline{f}(\Delta) = 
\begin{pmatrix}
 0 & -1 & 0 & 0\\
- 1 & 0 & 0 & 0\\
 0 & 0 & 0 & -1\\
 0 & 0 & -1 & 0
\end{pmatrix}.
\end{equation*}
One checks that 
\begin{equation}\label{surj1}
x_{\beta} = M_5 \, , \qquad 
x_{\alpha + \beta} = M_1 M_3^{-1} M_5 \, ,\qquad 
x_{2\alpha + \beta} = M_1 \, ,
\end{equation}
\begin{equation}\label{surj2}
x_{\alpha} = M_5^{-1} M_4^{-1} M_1^{-1} M_3 M_5^{-1} M_4 M_5 \, , 
\end{equation}
\begin{equation}\label{surj3}
w_{\alpha} = (M_4 M_5)^3 M_{\Delta}\, , \qquad
w_{\beta} = (M_4 M_5 M_4)^{-1} \, .
\end{equation}
This shows that Behr's generators all belong to the image of
the homomorphism~$\overline{f}$,
which proves the surjectivity of~$\overline{f}$.

Each among the 18 relations in Behr's presentation yields a generator of the kernel of~$\overline{f}$
as follows: we first write each relation as an equality $r_i = 1$, where $r_i$ is a word in Behr's generators
and their inverses ($1Ê\leq i \leq 18$); 
then in view of\,\eqref{surj1}--\eqref{surj3}, 
we replace in~$r_i$ each generator by the following corresponding lift in~$B_6$:
\begin{equation*}
x_{\beta} \leadsto \sigma_5 \, , \qquad 
x_{\alpha + \beta} \leadsto \sigma_1 \sigma_3^{-1} \sigma_5 \, ,\qquad 
x_{2\alpha + \beta} \leadsto \sigma_1 \, ,
\end{equation*}
\begin{equation*}
x_{\alpha} \leadsto \sigma_5^{-1} \sigma_4^{-1} \sigma_1^{-1} \sigma_3 \sigma_5^{-1} \sigma_4 \sigma_5 \, , 
\end{equation*}
\begin{equation*}
w_{\alpha} \leadsto (\sigma_4 \sigma_5)^3 \Delta \, ,\qquad 
w_{\beta} \leadsto (\sigma_4 \sigma_5 \sigma_4)^{-1} \, .
\end{equation*}
In this way, from each relation~$r_i=1$ we obtain an element~$\gamma_i \in B_6$,
which belongs to the kernel of~$\overline{f}$. It follows from~\cite[Satz]{Be}
that the 18 elements $\gamma_1, \ldots, \gamma_{18}$ generate a subgroup
whose normal closure is the kernel of~$\overline{f}$.

It is easy to check that $\gamma_3$, $\gamma_4$, $\gamma_5$, $\gamma_6$, $\gamma_8$,
$\gamma_9$, $\gamma_{11}$, $\gamma_{12}$, $\gamma_{15}$, $\gamma_{16}$, and $\gamma_{18}$
are all equal to the trivial element of~$B_6$. 
Therefore, the kernel of~$\overline{f}$ is the normal closure of the subgroup generated by the remaining elements
$\gamma_1$, $\gamma_2$, $\gamma_7$, $\gamma_{10}$, $\gamma_{13}$, $\gamma_{14}$,~$\gamma_{17}$.

Let us now handle these seven elements one by one.
Behr's Relation\,(1) is equivalent to $r_1 = 1$, where 
$r_1 = x_{2\alpha+\beta}^{-1}x_{\alpha+ \beta}^{-1}  x_{\alpha} x_{\beta}  x_{\alpha}^{-1} x_{\beta}^{-1}$.
We thus have
\begin{equation}\label{g1}
\gamma_1 
= \sigma_1^{-1} (\sigma_5^{-1} \sigma_3 \sigma_1^{-1} ) 
(\sigma_5^{-1} \sigma_4^{-1} \sigma_1^{-1} \sigma_3 \sigma_5^{-1} \sigma_4 \sigma_5)
\sigma_5
(\sigma_5^{-1} \sigma_4^{-1} \sigma_1 \sigma_3^{-1} \sigma_5 \sigma_4 \sigma_5) \sigma_5^{-1} \, .
\end{equation}

Relation\,(2) in~\cite{Be} is equivalent to $r_2 = 1$, where 
$r_2 = x_{2\alpha+\beta}^{-2} x_{\alpha} x_{\alpha+ \beta}  x_{\alpha}^{-1} x_{\alpha+ \beta}^{-1}$.
Therefore,
\begin{multline}\label{g2}
\gamma_2 = \sigma_1^{-2}
(\sigma_5^{-1} \sigma_4^{-1} \sigma_1^{-1} \sigma_3 \sigma_5^{-1} \sigma_4 \sigma_5) \times \\
{}\times (\sigma_1 \sigma_3^{-1} \sigma_5)
(\sigma_5^{-1} \sigma_4^{-1} \sigma_1 \sigma_3^{-1} \sigma_5 \sigma_4 \sigma_5)
(\sigma_1^{-1} \sigma_3 \sigma_5^{-1}) \, .
\end{multline}
Using the braid applet of~\cite{DF}, we deduce from~\eqref{g1} and~\eqref{g2} that 
\begin{equation}\label{g1g2}
\gamma_1 = \sigma_3 \gamma_2\sigma_3^{-1} \, .
\end{equation}

Behr's Relation\,(7) is equivalent to $r_7 = 1$ with $r_7 = w_{\alpha} w_{\beta}^2 w_{\alpha} w_{\beta}^{-2}$.
We thus have
\begin{eqnarray*}
\gamma_7 
& = & (\sigma_4 \sigma_5)^3 \Delta (\sigma_4\sigma_5\sigma_4)^{-2}
(\sigma_4 \sigma_5)^3 \Delta (\sigma_4 \sigma_5 \sigma_4)^2 \\
& = & (\sigma_4 \sigma_5)^3 \Delta (\sigma_4\sigma_5\sigma_4)^{-2}
(\sigma_4 \sigma_5 \sigma_4)^2 \Delta (\sigma_4 \sigma_5)^3 \\
& = & (\sigma_4 \sigma_5)^3 \Delta^2 (\sigma_4 \sigma_5)^3 \, .
\end{eqnarray*}
Since $\Delta^2$ is central, as is well known (see~\cite{Bi2,KT}), we obtain
\begin{equation}\label{g7}
\gamma_7 = (\sigma_4 \sigma_5 )^6 \Delta^2 \, .
\end{equation}

Behr's Relation\,(10), which is equivalent to $w_{\beta}^{-4} = 1$, yields
\begin{equation}\label{g10}
\gamma_{10} = (\sigma_4 \sigma_5 \sigma_4 )^4  = (\sigma_4 \sigma_5)^6 \, .
\end{equation}

Relation\,(13) in~\cite{Be} is equivalent to $r_{13} = 1$, where 
$r_{13} = w_{\alpha} x_{\alpha+ \beta}  w_{\alpha}^{-1} x_{\alpha+ \beta}$.
Therefore, 
\begin{equation*}
\gamma_{13} 
= (\sigma_4 \sigma_5)^3 \Delta
(\sigma_1 \sigma_3^{-1} \sigma_5)
\Delta^{-1}(\sigma_4 \sigma_5)^{-3}
(\sigma_1 \sigma_3^{-1} \sigma_5) \, .
\end{equation*}
In view of the relations $\Delta \sigma_i = \sigma_{6-i} \Delta$ ($i= 1, \ldots, 5$), we obtain
\begin{equation}\label{g12}
\gamma_{13} 
=  (\sigma_4 \sigma_5)^3 
(\sigma_1 \sigma_3^{-1} \sigma_5) (\sigma_4 \sigma_5)^{-3}
(\sigma_1 \sigma_3^{-1} \sigma_5)  \, .
\end{equation}
Using again the braid applet~\cite{DF}, we find
\begin{equation}\label{g13}
\gamma_{13} = \gamma_2^{-1} \, .
\end{equation}

Relation\,(14) in~\cite{Be} is equivalent to $r_{14} = 1$, where 
$r_{14} = x_{\alpha} w_{\beta}^{-1} x_{\alpha+ \beta}^{-1} w_{\beta}$.
Therefore,
\begin{equation*}
\gamma_{14} 
= (\sigma_5^{-1} \sigma_4^{-1} \sigma_1^{-1} \sigma_3 \sigma_5^{-1} \sigma_4 \sigma_5)
(\sigma_4 \sigma_5 \sigma_4)
(\sigma_1^{-1} \sigma_3 \sigma_5^{-1})
(\sigma_4^{-1} \sigma_5^{-1} \sigma_4^{-1})
\, .
\end{equation*}
We similarly find 
\begin{equation}\label{g14}
\gamma_{14} = \gamma_2 = \gamma_{13}^{-1} \, .
\end{equation}

Relation\,(17) in~\cite{Be} 
is equivalent to $r_{17} = 1$ with
$r_{17} = w_{\alpha} x_{\alpha}  w_{\alpha}^{-1} x_{\alpha} w_{\alpha} x_{\alpha}$.
Thus, 
\begin{multline*}
\gamma_{17} 
 = (\sigma_4 \sigma_5)^3 \Delta
(\sigma_5^{-1} \sigma_4^{-1} \sigma_1^{-1} \sigma_3 \sigma_5^{-1} \sigma_4 \sigma_5)
\Delta^{-1}(\sigma_4 \sigma_5)^{-3} \times\\
\times ( \sigma_5^{-1} \sigma_4^{-1} \sigma_1^{-1} \sigma_3 \sigma_5^{-1} \sigma_4 \sigma_5 )
(\sigma_4 \sigma_5)^3 \Delta 
( \sigma_5^{-1} \sigma_4^{-1} \sigma_1^{-1} \sigma_3 \sigma_5^{-1} \sigma_4 \sigma_5) \, .
\end{multline*}
Letting both $\Delta$ jump to the right and letting the leftmost one merge with~$\Delta^{-1}$,
we obtain
\begin{multline}\label{g17}
\gamma_{17} 
 =  (\sigma_4 \sigma_5)^3 
 (\sigma_1 \sigma_2 \sigma_1^{-1} \sigma_3 \sigma_5^{-1} \sigma_2^{-1} \sigma_1^{-1})
(\sigma_4 \sigma_5)^{-3} \times\\
\times (\sigma_5^{-1} \sigma_4^{-1} \sigma_1^{-1} \sigma_3 \sigma_5^{-1} \sigma_4 \sigma_5)
(\sigma_4 \sigma_5)^3 
(\sigma_1 \sigma_2 \sigma_1^{-1} \sigma_3 \sigma_5^{-1}
 \sigma_2^{-1} \sigma_1^{-1}) \Delta
\, .
\end{multline}

It follows from~\eqref{g1g2}, \eqref{g13}, \eqref{g14}
that the kernel of~$\overline{f}$ is the normal closure of the subgroup generated
by $\gamma_7$, $\gamma_{10}$, $\gamma_{13}$, and $\gamma_{17}$. 
Now, by~\eqref{g7}--\eqref{g12} we have $\gamma_7 = \alpha^2 \Delta^2$, $\gamma_{10} = \alpha^2$, and
\begin{equation*}
\gamma_{13} 
= \alpha \gamma \alpha^{-1} \gamma 
= (\alpha\gamma)^2 \, \gamma^{-1} (\alpha^2)^{-1} \gamma \, ,
\end{equation*}
where $\Delta$, $\alpha$, $\gamma$ have been defined in the statement of the theorem.
Hence, the kernel of~$\overline{f}$ is generated as a normal subgroup 
by $\Delta^2$, $\alpha^2$, $(\alpha\gamma)^2$, and~$\gamma_{17}$.
To complete the proof of the theorem, we consider the normal subgroup~$N$ of~$B_6$
generated by $\Delta^2$, $\alpha^2$, and~$(\alpha\gamma)^2$, and we show
that $\gamma_{17}$ is conjugate to~$\alpha\beta$ modulo~$N$.

Let us now prove this. From\,\eqref{g17} we first derive
\begin{equation*}
\gamma_{17}
= \sigma_1 \sigma_2 \alpha \gamma^{-1} \sigma_2^{-1} \sigma_1^{-1}
\alpha^{-1} \sigma_5^{-1} \sigma_4^{-1} \gamma^{-1} \sigma_4 \sigma_5 \alpha
\sigma_1 \sigma_2 \gamma^{-1} \sigma_2^{-1} \sigma_1^{-1} \Delta \, .
\end{equation*}
Now, $(\alpha \gamma)^2 \equiv 1$ modulo~$N$. Hence, $\gamma^{-1} \equiv \alpha \gamma\alpha$.
Therefore, in view of this and of~$\alpha^2 \equiv 1$, we obtain
\begin{eqnarray*}
\gamma_{17}
& \equiv & \sigma_1 \sigma_2 \alpha^2 \gamma \alpha \sigma_2^{-1} \sigma_1^{-1}
\alpha^{-1} \sigma_5^{-1} \sigma_4^{-1} \alpha \gamma\alpha \sigma_4 \sigma_5 \alpha
\sigma_1 \sigma_2 \alpha \gamma\alpha^{-1} \alpha^2 \sigma_2^{-1} \sigma_1^{-1} \Delta \\
& \equiv & \sigma_1 \sigma_2 \sigma_1 \sigma_3^{-1} \sigma_2^{-1}  \sigma_1^{-1}
\sigma_5\sigma_5^{-1} \sigma_4^{-1} \alpha \gamma \sigma_4 \sigma_5
\sigma_1 \sigma_2 \alpha \gamma\alpha^{-1} \sigma_2^{-1} \sigma_1^{-1} \Delta \\
& \equiv & \sigma_1 \sigma_2 \sigma_1 \sigma_3^{-1} \sigma_2^{-1}\sigma_1^{-1} \sigma_4^{-1} 
\alpha \sigma_1 \sigma_3^{-1} \sigma_5 \sigma_4 \sigma_5
\sigma_1 \sigma_2 \alpha \gamma\alpha^{-1} \sigma_2^{-1} \sigma_1^{-1} \Delta \\
& \equiv & (\sigma_1 \sigma_2 \sigma_1 \sigma_3^{-1} \sigma_2^{-1} \sigma_4^{-1} )
\alpha \sigma_3^{-1} \sigma_1 \sigma_2  \sigma_5 \sigma_4 \sigma_5 \alpha \gamma
(\alpha^{-1} \sigma_2^{-1} \sigma_1^{-1} \Delta) 
\end{eqnarray*}
modulo~$N$.
Using the applet~\cite{DF}, we obtain
\begin{equation*}
\alpha^{-1} \sigma_2^{-1} \sigma_1^{-1} \Delta =
\sigma_3 \sigma_4 \sigma_5 \sigma_2 \sigma_1 \sigma_2 \sigma_3 
(\sigma_1 \sigma_2 \sigma_1 \sigma_3^{-1} \sigma_2^{-1}\sigma_4^{-1})^{-1} \, .
\end{equation*}
Therefore, $\gamma_{17} \equiv 
(\sigma_1 \sigma_2 \sigma_1 \sigma_3^{-1} \sigma_2^{-1}\sigma_4^{-1}) \, \gamma' \, 
(\sigma_1 \sigma_2 \sigma_1 \sigma_3^{-1} \sigma_2^{-1}\sigma_4^{-1})^{-1}$,
where
\begin{eqnarray*}
\gamma'
& = & \alpha \sigma_3^{-1} \sigma_1 \sigma_2  \sigma_5 \sigma_4 \sigma_5 \alpha \gamma
\sigma_3 \sigma_4 \sigma_5 \sigma_2 \sigma_1 \sigma_2 \sigma_3 \\
& = & \alpha \sigma_3^{-1}( \sigma_1 \sigma_2  \sigma_5 \sigma_4 \sigma_5 \alpha \sigma_5
\sigma_1 \sigma_3^{-1} \sigma_3 \sigma_4 \sigma_5 \sigma_2 \sigma_1 \sigma_2) \sigma_3 \\
& = & \alpha \sigma_3^{-1}( (\sigma_1 \sigma_2)  (\sigma_5 \sigma_4 \sigma_5 \alpha \sigma_5
\sigma_4 \sigma_5) (\sigma_1\sigma_2)^2) \sigma_3 \\
& = & \alpha \sigma_3^{-1}( (\sigma_1 \sigma_2)  \alpha^2 (\sigma_1\sigma_2)^2) \sigma_3 \\
& \equiv & \alpha \sigma_3^{-1} (\sigma_1 \sigma_2)^3 \sigma_3 = \alpha\beta\, .
\end{eqnarray*}
This completes the proof of Theorem\,\ref{thm-sp4}.
\end{proof}

\begin{remark}\label{rem-auto}
We may wonder whether the four elements $\Delta^2$, $\alpha^2$, $\alpha\beta$, $(\alpha\gamma)^2$
of the kernel of~$\overline{f}$ already belong to the kernel of~$f$.
This holds true for the central element~$\Delta^2$: 
indeed, $f(\Delta^2) = \id$ by Proposition~\ref{vanish}.
The other three elements 
map under~$f$ to non-inner automorphisms of~$F_4$.
For instance, the automorphism $f(\alpha^2)$ is the following: 
it fixes~$a_1$ and~$b_1$, and on~$a_2$ and~$b_2$
we have
\begin{equation*}
a_2  \mapsto  (a_2^{-1} b_2 a_2 b_2^{-1}) a_2 (a_2^{-1} b_2 a_2 b_2^{-1})^{-1}
\quad\text{and}\quad
b_2  \mapsto  (a_2^{-1} b_2 a_2 ) b_2 (a_2^{-1} b_2 a_2 )^{-1} \, .
\end{equation*}
\end{remark}

\section*{Acknowledgments}

I am grateful to Etienne Ghys and Vladimir Turaev for their hints,
and to the referees for useful remarks.
I am also indebted to Pierre Guillot who computed certain automorphisms $f(\beta) \in \Aut(F_4)$ using GAP.

\end{document}